\documentclass[12pt,reqno]{amsart}

\usepackage[margin=1in]{geometry}

\usepackage{amsmath}%
\usepackage{amsfonts}%
\usepackage{amssymb}%
\usepackage{float}
\usepackage{amsthm}
\usepackage{graphicx}
\usepackage{float}
\usepackage{color}
\usepackage{hyperref}
\usepackage[normalem]{ulem}
\usepackage{pdfsync}
\usepackage{float}

\numberwithin{equation}{section}

\newtheorem{theorem}{Theorem}
\newtheorem{lemma}{Lemma}
\newtheorem{definition}{Definition}
\newtheorem{observation}{Observation}
\newtheorem{algorithm}{Algorithm}

\newcommand{\ve}{\varepsilon}
\renewcommand{\epsilon}{\varepsilon}

\newtheorem{remark}{Remark}

\makeatletter
\renewcommand{\maketag@@@}[1]{\hbox{\m@th\normalsize\normalfont#1}}%
\makeatother

\makeatletter
\newcommand\footnoteref[1]{\protected@xdef\@thefnmark{\ref{#1}}\@footnotemark}
\makeatother

\title{Uniqueness and traveling waves in a cell motility model}
\author{Matthew S. Mizuhara}
\thanks{MSM: Department of Mathematics, The Pennsylvania State University, University Park, PA 16802, USA. The work of MSM was supported by the Department of Defense (DoD) through the National Defense Science \& Engineering Graduate Fellowship (NDSEG) Program. He also received partial support from NSF grants DMS-1106666 and DMS-1405769.} 
\email{msm344@psu.edu}
\author{Peng Zhang}
\thanks{PZ: Institute of Natural Sciences and Department of Mathematics, Shanghai Jiao Tong University; Key Laboratory of Scientific and Engineering Computing (Shanghai Jiao Tong University), Ministry of Education, 800 Dongchuan Road, 200240, Shanghai, China. The work of PZ was partially supported by the {National Natural Science Foundation of China grant 11471214 and the One Thousand Plan of China for young scientists.}}
\email{zhangpengmath@sjtu.edu.cn}

\begin{document}


\maketitle

\begin{abstract}
	    We study a non-linear and non-local evolution equation for curves obtained as the sharp interface limit of 
	    a phase-field model for crawling motion of eukaryotic cells on a substrate.
	    We establish uniqueness of solutions to the sharp interface limit equation in the so-called subcritical parameter regime.  The proof relies on a Gr\"onwall estimate for a specially chosen weighted $L^2$ norm. 
	    
	     Next, as persistent motion of crawling cells is of central interest to biologists we study the existence of traveling wave solutions.  We prove that traveling wave solutions exist in the supercritical parameter regime provided the non-linear term of the sharp interface limit equation possesses certain asymmetry (related, e.g., to myosin contractility).
	    
	    Finally, we numerically investigate traveling wave solutions and simulate their dynamics. Due to non-uniqueness of solutions of the sharp interface limit equation we simulate a related, singularly perturbed PDE system which is uniquely solvable. Our simulations predict instability of traveling wave solutions and capture both bipedal wandering cell motion as well as rotating cell motion; these behaviors qualitatively agree with recent experimental and theoretical findings.
\end{abstract}

\section{Introduction}

%

Motility of eukaryotic cells has been long studied by both biologists and mathematicians due to its importance in various biological processes such as wound healing \cite{Moh03} and the immune response \cite{KruBarGer16}.   
The motion of these crawling cells is driven by an evolving cytoskeleton which creates protrusion forces against the surrounding cell membrane. Various modeling efforts have studied a large variety of phenomena including persistent motility (see, e.g., \cite{BarLeeAllTheMog15, Mar15, RecPutTru15, Sha10, Tjh15, Zie12}), turning cells \cite{CamZhaLiLevRap16}, spontaneous division of cells \cite{Gio14}, symmetry breaking and 3D models of cell motility \cite{BarLeeAllTheMog15, Haw11,Tjh15}, hydrodynamic interactions with the cell cytoplasm \cite{Mar15}, interaction with non-homogeneous substrates \cite{Zie13}, and collective cell migration \cite{CamZimLevRap16, LobZieAra15}. Specifically, recent phase-field models have had great success in capturing various modes of motility as well as their dependence on physical parameters.  The main feature of phase-field models is a diffuse interface which approximates the location of the cell membrane separating two phases (e.g., interior and exterior of cell). As the width of the diffuse interface tends to zero, one recovers the so-called sharp interface limit equation: a geometric evolution equation for planar curves which is amenable to analysis and subsequent numerics.


%


In this work we analytically and numerically study the following evolution equation for a family of curves $\Gamma(t)$:
\begin{equation}\label{interface}
V(s,t) = \kappa(s,t) + \Phi_\beta(V(s,t)) - \lambda(t).
\end{equation}
Here, $s\in I$ is any parametrization of $\Gamma(t)$, $V(s,t)$ denotes the inward pointing normal velocity of the curve $\Gamma(t)$ at location $s$, $\kappa(s,t)$ is the signed curvature of the curve $\Gamma(t)$ at location $s$, $\Phi_\beta(\cdot)$ is a known non-linear function depending on a physical parameter $\beta\geq 0$, and $\lambda(t)$ is chosen so that the area enclosed by $\Gamma(t)$ is constant for all $t$:
\begin{equation}
\lambda(t) = \frac{1}{|\Gamma(t)|}\int_\Gamma \Phi_\beta(V(s,t)) + \kappa(s,t) ds,
\end{equation}
where $|\Gamma(t)|$ denotes the length of the curve $\Gamma(t)$.

Equation \eqref{interface} is volume preserving curvature motion with additional non-linearity due to $\Phi_\beta$. The problem of mean curvature type motion was extensively studied by mathematicians from both PDE and geometry communities for several decades, see, e.g., \cite{Bra78,GagHam86,Gra87}. Furthermore viscosity solution techniques have been efficiently applied in the PDE analysis of such problems \cite{CheGigGot91, EvaSpr93}. Analysis of mean curvature motions with volume preservation constraints was continued in, e.g., \cite{Che93,Ell97, Esc98,Gag86}. Existence and uniqueness results of modified mean curvature motions have additionally been recently studied \cite{Bon00,Che93, Ell97}. Analogous problems have resurfaced in the novel context of biological cell motility problems, after a phase-field model was introduced in \cite{Zie12} and its sharp interface limit \eqref{interface} derived in \cite{BerPotRyb16}. In a certain so-called subcritical regime, existence of solutions and traveling waves of \eqref{interface} were studied in \cite{MizBerRybZha15}.  In this work we prove uniqueness of solutions in the subcritical regime. We additionally study traveling waves in the supercritical regime both analytically and numerically.

Here we summarize the results of each section. First, we provide the necessary biological and mathematical background related to the model of crawling cell motility.  In Section 2 we establish uniqueness of solutions to \eqref{interface} in the subcritical regime. The main difficulties are the non-linearity and non-locality in the sharp interface limit equation; we will establish a Gr\"onwall inequality on a specially chosen weighted $L^2$ norm which allows for necessary estimates on both the non-linear and non-local terms.  

In Section 3 we study traveling wave solutions in the supercritical regime. We prove that if $\Phi_\beta$ is symmetric (i.e. $\Phi_\beta(V)=\Phi_\beta(-V)$) then the only traveling wave solution is a stationary circle. This requires new estimates due to lack of monotonicity of the operator (c.f., subcritical regime \cite{MizBerRybZha15}).  We also prove a sufficient condition on $\Phi_\beta$ for the existence of non-trivial traveling wave solutions. 

In Section 4 we present numerical results. We first introduce an algorithm which searches for and constructs the profiles of traveling wave solutions. We next study the stability of traveling wave solutions via dynamic simulations. Due to non-uniqueness of solutions to \eqref{interface} we require introduction of a regularization term. Our simulations capture both wandering cells \cite{Zie13} as well as a recent experimentally observed phenomenon of rotating cells \cite{Lou15}. Moreover we observe the switch between these two regimes depending on parameter values.

\subsection{Biological background}

 
 Fish keratocyte cells are commonly used in experiments due to their ability to undergo persistent motion with essentially constant shape and direction for many cell lengths after an initial polarization. Moreover they may spontaneously polarize from non-motile, circular states and begin traveling persistently even without external stimuli \cite{BarLeeAllTheMog15}.  Crawling keratocytes on flat substrates are several microns in length and width, but have heights of only $\sim$.1-.2 $\mu$m, making them amenable to 2D modeling and simulation \cite{MogKer09}.  

  The biological mechanisms involved in crawling cell motion are complicated and provide an ongoing source of research.  An overview for such mechanisms can be found in, e.g., \cite{Mog09}. We recall the following key components: consider the cell membrane to be an inextensible bag and assume that a directional preference has been established. A crawling cell has the ability to maintain self-propagating motion via internal forces generated by {\em actin polymerization}. Actin monomers are polarized so as to bind together forming filaments which create a dense network at the leading edge of the cell, known as the lamellipod. The cell's front protrudes via growth of actin filaments at the leading edge and degradation of the filaments towards the interior of the cell, a process known as {\em actin treadmilling}. Adhesions of the cell to the substrate creates traction forces which generate cell motion.
  
 {\em Myosin} is a molecular motor which interacts with actin filaments creating contractile forces. Actomyosin interaction leads to contraction of the rear part of the cell in steadily moving cells. Recent models also suggest that myosin-driven contraction of the cytoskeleton is sufficient to drive persistent motion of the cell \cite{RecPutTru15} or spontaneously switch a cell from a symmetric, non-motile state to asymmetric, motile states \cite{BarLeeAllTheMog15}.
  
  There are various modes of motility which are observed in both experiments and simulations. First, keratocytes may move with essentially the same shape and direction for many cell lengths. Such persistent motion is described mathematically by traveling wave solutions.  In \cite{BarAllJulThe10} so-called ``bipedal motion'' was observed wherein the cell exhibited periodic lateral out-of-phase oscillations. These oscillations were explained by the effect of  elastic coupling between the rear and front of the cell membrane. This motion was additionally captured by phase-field simulations in \cite{LobZieAra14}.
  
  A final mode, so-called ``rotating cells,''  was recently observed  in \cite{Lou15} where cells remain essentially stationary but experience laterally periodic protrusions of the membrane. In this case, these transient protrusions are generated by actin polymerization fronts which experience short protrusion lifetimes. This decrease is due to expression of a particular kinase (MLCK) leading to an increase of myosin activity in the cell's lamellipod ultimately limiting actin protrusions globally.

\subsection{Phase-field model of crawling cell motion and sharp interface limit}\label{pf_intro}
Equation \eqref{interface} is the sharp interface limit of the 2D phase-field model of crawling cell motility first studied in \cite{BerPotRyb16} and subsequently in \cite{MizBerRybZha15}.

The phase-field model consists of a system of two PDEs: one for the phase field function and another for the orientation vector due to polymerization of actin filaments inside the cell. 
Let $\Omega\subset \mathbb{R}^2$ be a smooth bounded domain, then
\begin{equation}\label{eq1}
\frac{\partial \rho_\ve}{\partial t}=\Delta \rho_\ve
-\frac{1}{\ve^2}W^{\prime}(\rho_\ve)
- P_\ve\cdot \nabla \rho_\ve +\lambda_\ve(t)\;\;
  \text{ in }\Omega,
\end{equation}
\begin{equation}\label{eq2}
\frac{\partial P_\ve}{\partial t}=\ve\Delta P_\ve -\frac{1}{\ve}P_\ve
 -\beta \nabla \rho_\ve\;\;
 \text{ in } \Omega,
\end{equation}
where
\begin{equation*} \label{lagrange}
\lambda_\ve(t)=\frac{1}{|\Omega|}\int_\Omega\left(\frac{1}{\ve^2}W^\prime(\rho_\ve)
+ P_\ve\cdot \nabla \rho_\ve \right)\, dx
\end{equation*}
is a {\em Lagrange multiplier} term responsible for total volume preservation of $\rho_\ve$, and $W$ is a double well potential (e.g., $W(z)=\frac{1}{4}z^2(1-z)^2$). The shape of $W$ is determined, e.g., by effects of myosin motors (see Remark \ref{rem:myosin}). The quantity $\beta\geq 0$ is a physical parameter which characterizes actin polymerization strength and rate as well as adhesion strength. System \eqref{eq1}-\eqref{eq2} is a simplified version of a more general phase-field model introduced in \cite{Zie12}. These simplifications retain key features of qualitative behavior and is suitable for asymptotic analysis.


	Details of the model terms and sharp interface limit analysis can be found in \cite{BerPotRyb16}. Here we briefly describe the terms of the phase-field model.
  The phase-field parameter $\rho_\ve\colon \Omega\rightarrow \mathbb{R}$, roughly speaking, takes values 1 and 0 inside and outside, respectively, a subdomain $D_\ve(t)\subset \Omega$ occupied by the moving cell with a thin interface layer of width $O(\ve)$.  The first two terms on the right hand side of \eqref{eq1} are recognized from the classic Allen-Cahn equation. Together they enforce preservation of a thin interface layer and drive the interface layer to move by curvature motion, e.g., modeling surface tension.  The drift term $P_\ve \cdot \nabla \rho_\ve$ models advection due to actin polymerization. The vector valued function $P_\ve \colon \Omega\rightarrow \mathbb{R}^2$ models the polar orientation of actin filaments inside the cell. The magnitude of $P_\ve$ is a measure of the degree of ordering of the filaments. In addition to spatial diffusion we include degradation of actin filaments via the exponential decay term $-\ve^{-1} P_\ve$ and polymerization of actin filaments at the boundary via the source term $-\beta \nabla \rho_\ve$. 
  
  In the limit $\ve \to 0$, the $\rho_\ve$ converge to time-dependent characteristic functions: $D(t) = \lim_{\ve \to 0} D_\ve (t)$ with sharp boundary $\Gamma(t):=\partial D(t)$. 
 Indeed, given a closed, non self-intersecting curve $\Gamma(0)\subset \mathbb{R}^2$, consider the initial profile
  \begin{equation*}\label{phaseIC}
  \rho_\ve(x,0) = \theta_0\left(\frac{\operatorname{dist}(x,\Gamma(0))}{\ve}\right)
  \end{equation*}
  where $\operatorname{dist}(x,\Gamma(0))$ is the signed distance from the point $x$ to the curve $\Gamma(0)$ and $\theta_0=\theta_0(z)$ solves:
    \begin{equation}\label{standingwave}
    \theta_0''(z) = W'(\theta_0(z)), \;\;\; \theta_0(-\infty) = 0,\; \theta_0(\infty) = 1.
    \end{equation} 
    For example, if $W$ is the classical Allen-Cahn potential: $W(z) = \frac{1}{4} z^2(1-z)^2$, then $\theta_0$ is the standing wave solution:
  \begin{equation*}\label{steadysolution}
  \theta_0(z) = \frac{1}{2}\left(\tanh\left(\frac{z}{2\sqrt{2}}\right)+1\right).
  \end{equation*}
   It was shown in  \cite{BerPotRyb16} that as $\ve \to 0$ (i.e., in the sharp interface limit),  $\rho_\ve(x,t)$ has the asymptotic form
  \begin{equation*}\label{asymptoticform}
  \rho_\ve(x,t) = \theta_0\left(\frac{\operatorname{dist}(x,\Gamma(t))}{\ve}\right)+ O(\ve)
  \end{equation*}
  where the family of curves $\Gamma(t)$ evolves by \eqref{interface} with $\Phi_\beta$ defined by
   \begin{equation}\label{phi}
    \Phi_\beta(V) := \int_\mathbb{R} \psi(z; V)(\theta_0'(z))^2 dz
    \end{equation}
    where $\psi(z)=\psi(z;V)$ is the solution of
    \begin{equation}\label{phi1}
    \psi''(z)+V\psi'(z)-\psi(z)- \beta\theta_0' =0, \;\; \psi(\pm \infty)=0.
    \end{equation}
    Note that $\Phi_\beta$ depends on the function $\theta_0$ and thus on the potential $W$. In particular, if $W$ is a symmetric, equal double-well potential, e.g., the Allen-Cahn potential $W(z) = \frac{1}{4} z^2(1-z)^2$, then $\Phi_\beta$ is even: $\Phi_\beta(V)= \Phi_\beta(-V)$ for all $V\in \mathbb{R}$. If $W$ possesses some asymmetry then $\Phi_\beta$ may be asymmetric.
    \begin{remark}\label{rem:myosin}
    Asymmetric potential wells arise in the original phase-field model \cite{Zie12} due to interaction of myosin with actin. Indeed, in \cite{Zie12} the potential is of the form $W'(z) = (1-z)(\delta-z)z$, where
    \begin{equation}
    \delta := \frac{1}{2} + \mu\left(\int \rho dx - V_0\right) - \sigma|P_\ve|^2.
    \end{equation}
    The integral term represents volume preservation with stiffness parameter $\mu$, represented as the Lagrange multiplier in \eqref{eq1}-\eqref{eq2}. The term $\sigma |P_\ve|^2$ models contractile stress due to myosin motors, see \cite{Zie12} for details. It follows that if $\delta \neq \frac{1}{2}$, then the double-well potential is asymmetric.
    \end{remark}

  The analysis and the behavior of solutions of \eqref{interface} crucially depends on the parameter $\beta$ in \eqref{phi1}. Define
  \begin{equation}
  	\beta_{crit} := \sup\{\beta \mid \|\Phi_\beta'\|_{L^\infty(\mathbb{R})}<1\}.
  \end{equation}

  The subcritical regime consists of all $\beta<\beta_{crit}$. In this regime, $V-\Phi_\beta(V)$ is monotone and so \eqref{interface} is uniquely solvable for $V$.
    Note the particular case $\beta=0$ in \eqref{eq1}-\eqref{eq2} leads to $P_\ve \equiv 0$ reducing the PDE system to the volume preserving Allen-Cahn equation. Properties of this equation were studied in \cite{Che10,Gol94} and it was proved that the sharp interface limit as $\ve\to 0$ recovers volume preserving mean curvature motion: $V=\kappa-\frac{1}{|\Gamma|}\int_{\Gamma} \kappa ds$. 
     Short time existence of solution curves for any subcritical $\beta<\beta_{crit}$ was proved in \cite{MizBerRybZha15} for a general class of initial data.

    In the supercritical regime, $\beta>\beta_{crit}$, complicated phenomena such as non-uniqueness and hysteresis arise, even in 1D (see \cite{BerPotRyb16}). The non-uniqueness of solutions of \eqref{interface} in the supercritical regime is the result of ``losing information'' in the sharp interface limit. To resolve this we consider an intermediate system (between the full phase-field model and the sharp interface limit) which still defines a geometric evolution but remains uniquely solvable. In 1D the system is:
        \begin{align}
        c_0 V_\ve &= \int (\theta_0'(z))^2 f_\ve(z,t)dz - F(t) \label{stability_eqn1} \\
        \ve \partial_t f_\ve  &= \partial_{zz}^2 f_\ve + V_\ve \partial_z f_\ve -f_ve- \beta \theta_0', \;\; f_\ve(\pm \infty) = 0 \label{stability_eqn2}
        \end{align}
        where $F(t)$ is a given function replacing the effects of curvature and the Lagrange multiplier terms.
        
         This intermediate system will be beneficial for numerical simulations and linear stability analysis. In particular it is proved in \cite{BerPotRyb16} that a necessary condition for stability of a steady state $(V^*,f^*)$ solving \eqref{stability_eqn1}-\eqref{stability_eqn2} is that $\Phi_\beta'(V^*) < c_0$. It is expected (and verified numerically in 1D) that this is also a sufficient condition. We thus call $V^*$ such that $\Phi_\beta'(V^*)<c_0$ a {\em stable velocity} and an {\em unstable velocity} otherwise.
    
    In 2D the analogous intermediate system is
  
  \begin{align}
  V_\epsilon(s,t) &= \kappa_\epsilon(s,t) + \int_{\mathbb{R}} A_\epsilon(s,z,t) (\theta_0'(z))^2 dz - \lambda_\epsilon(t) \label{numeric_veqn1} \\
  \epsilon \partial_t A_\epsilon(s,z,t) &= \partial_{zz}A_\epsilon + V_\epsilon\partial_z A_\epsilon  -A_\epsilon - \beta\theta_0'(z),\;\; A_\epsilon (\pm \infty) = 0. \label{numeric_aeqn1}
  \end{align}
  
  It is clear that the formal limit $\ve \to 0$ yields \eqref{interface}. Passing to the limit rigorously has been done in 1D in \cite{BerPotRyb16}.

\section{Uniqueness}

In \cite{MizBerRybZha15} it was proved that if $\|\Phi_\beta'\|_{L^\infty(\mathbb{R})}<1$ (subcritical regime) then given any initial curve $\Gamma_0\in W^{1,\infty}$, there exists a time $T>0$ and a family of continuous in time curves $\Gamma(t)\in H^2$ evolving by \eqref{interface} for all $t\leq T$. Under additional regularity assumptions we prove uniqueness of solutions:




\begin{theorem}\label{thm_unique}
Let $\Phi_\beta\in W^{3,\infty}(\mathbb{R})$ satisfy $\|\Phi_\beta'\|_{L^\infty(\mathbb{R})}<1$ and let $\Gamma_0\in W^{3,\infty}$. There exists a time $T>0$ such that any solution $\Gamma(t)\in H^2$ satisfying \eqref{interface} on $[0,T]$ with initial condition $\Gamma(0)=\Gamma_0$ is unique.
\end{theorem}

\begin{remark}\label{rem_phase}
Recalling the definition of $\Phi_\beta$ above, it is clear that $\Phi_\beta\in C^\infty$ so the regularity assumptions on $\Phi_\beta$ in Theorem \ref{thm_unique} are natural and biologically relevant.
\end{remark}

In order to prove Theorem \ref{thm_unique} we first recast the geometric evolution equation as an equivalent PDE, see \cite{Ell97,MizBerRybZha15}.


Fix an initial curve $\Gamma_0\in W^{3,\infty}$. Let $
\tilde{\Gamma}\in C^\infty$ be a fixed reference curve close to $\Gamma_0$ parametrized by arc length $\sigma\in I$, with signed curvature $\kappa_0(\sigma)$ and inward pointing normal vector $\nu(s)$. Consider the tubular neighborhood
\begin{equation}\label{unbhd}
U_{\delta_0}:=\{x\in\mathbb{R}^2 | \operatorname{dist}(x,\tilde{\Gamma})<2\delta_0\}.
\end{equation}
For sufficiently small $\delta_0$, there exists a well defined diffeomorphism
\begin{equation}
Y\colon I \times (-2\delta_0,2\delta_0)\rightarrow U_{\delta_0}, Y(\sigma,u) := \tilde{\Gamma}(\sigma)+u(\sigma)\nu(\sigma).
\end{equation}
That is, for a function $u\colon I\times [0,T]\to [-\delta_0,\delta_0]$, periodic in $\sigma$, we have a well-defined correspondence
\begin{equation}\label{curvetou}
\Gamma(\sigma,t) = \tilde{\Gamma}(\sigma)+u(\sigma,t)\nu(\sigma)
\end{equation}
between $\Gamma(t)$ and $u(\cdot,t)$. 
In particular, there exists periodic $u_0 \in W^{3,\infty}(I)$ which parametrizes $\Gamma_0$:
\begin{equation}
\Gamma_0(\sigma) = \tilde{\Gamma}(\sigma) + u_0(\sigma)\nu(\sigma).
\end{equation}
We assume without loss of generality that $\delta_0$ chosen in \eqref{unbhd} is sufficiently small so that 
\begin{equation}
\delta_0\|\kappa_0\|_{L^\infty}<1.
\end{equation}

Using the Frenet-Serre formulas we express the normal velocity $V$ of $\Gamma(t)$ as
\begin{equation}\label{veqn}
V = V(u) = 
 \frac{1-u\kappa_0}{S}u_t
\end{equation}
where $S=S(u)= \sqrt{u^2_\sigma + (1-u\kappa_0)^2}$.
Curvature of $\Gamma(t)$ is given by
\begin{equation}\label{curvature}
\kappa(u) = 
\frac{1}{S^3} \Bigl((1-u \kappa_0) u_{\sigma\sigma} + 2\kappa_0 u_\sigma^2+(\kappa_0)_\sigma u_\sigma u  +
\kappa_0(1-u \kappa_0)^2\Bigr).
\end{equation}
Thus, we rewrite \eqref{interface} as the following PDE for $u$:
 \begin{equation}\label{deqn}
 \begin{aligned}
\frac{1-u\kappa_0}{S}u_t -\Phi_\beta\left(\frac{1-u\kappa_0}{S}u_t\right)
 &= \kappa(u)-\frac{1}{L[u]}\Bigl( \int_I \Phi_\beta\left(\frac{1-u \kappa_0}{S} u_t\right)Sd\sigma+2\pi\Bigr),
\end{aligned}
\end{equation}
where $L[u] := \int_I S(u) d\sigma$. Since the non-local term in \eqref{deqn} is independent of $\sigma$, for simplicity we may rewrite \eqref{deqn} as
 \begin{equation}\label{deqn1}
 \begin{aligned}
\frac{1-u\kappa_0}{S}u_t -\Phi_\beta\left(\frac{1-u\kappa_0}{S}u_t\right)
 &= \kappa(u)-\lambda_u(t)
\end{aligned}
\end{equation}
for some function $\lambda_u(t)\in L^\infty(\mathbb{R})$ which depends on $u$.

%

In \cite{MizBerRybZha15} short time existence of solutions $u$ to \eqref{deqn} was proved for a general class of initial conditions:
\begin{theorem}\label{thm_exist}
Let $u_0\in W^{1,\infty}(I)$. Then, there exists a time $T>0$ and a $u\in L^2(0,T; H^2(I))$ with $u_t \in L^2(0,T;L^2(I))$ so that $u$ solves \eqref{deqn} with $u(\sigma,0)=u_0(\sigma)$.
\end{theorem}

\begin{remark}
The time of uniqueness in Theorem \ref{thm_unique} may be less than the time of existence in Theorem \ref{thm_exist}; in fact, the proof of Theorem \ref{thm_unique} establishes uniqueness so long as $\Gamma(t)\in W^{3,\infty}$.
\end{remark}

We recall the following lemma proved in \cite{MizBerRybZha15} which will be useful in the sequel:
\begin{lemma}\label{lem_unif_bound}
Let $u$ be a solution of \eqref{deqn} (with initial value $u_0$) on the interval $[0,T]$ satisfying $\|u(t)\|_{L^\infty(I)}\leq \delta_0$ for all $t<T$. Then
\begin{equation}
\|u(t)\|_{L^\infty(I)} \leq \|u_0\|_{L^\infty(I)}+ Rt
\end{equation}
where $R\geq 0$ is a constant independent of $u_0$. Furthermore, the following inequality holds
\begin{equation}
\|u_\sigma(t)\|_{L^\infty(I)} \leq a_1(t),
\end{equation}
where $a_1(t)$ is the solution of
\begin{equation}
\frac{d a_1}{d t} = P_1 a_1^2 + Q_1 a_1 + R_1, \; a_1(0) = \|(u_0)_\sigma\|_{L^\infty(I)}
\end{equation}
(continued by $+\infty$ after the blow-up time) and $P_1, Q_1, R_1$ are positive constants independent of $u_0$.
\end{lemma}

We first utilize the additional smoothness on initial conditions and $\Phi_\beta$ in order to establish analogous uniform estimates on $u_{\sigma\sigma}$ and $u_{\sigma\sigma\sigma}$.

\begin{lemma}\label{lem_usigsig_est}
Suppose that $\Phi_\beta\in W^{3,\infty}(\mathbb{R})$ is a Lipschitz function satisfying 
\begin{equation}
\|\Phi_\beta'\|_{L^\infty(\mathbb{R})} < 1.
\end{equation}
Then for any $u_0 \in W^{3,\infty}(I)$ with $\|u_0\|_{L^\infty(I)}<\delta_0$, then there exists a time $T>0$ such that a solution $u$ of \eqref{deqn} exists on $[0,T]$ satisfying
\begin{equation}
\|u(t)\|_{W^{3,\infty}(I)} <\infty
\end{equation}
for all $t< T$.
\end{lemma}

\begin{proof}
Theorem \ref{thm_exist} and Lemma \ref{lem_unif_bound} imply that there exists a time $T>0$ so that any solution of \eqref{deqn} satisfies $\sup_{t\in[0,T]}\|u(t)\|_{L^\infty(I)}<\delta_0$ and $\sup_{t\in[0,T]}\|u_\sigma(t)\|_{L^\infty(I)}\leq C <\infty$, so to establish the claim we need only prove uniform estimates on $u_{\sigma\sigma}$ and $u_{\sigma\sigma\sigma}$. This will be done via a bootstrapping argument, establishing a maximum principle first for $u_{\sigma\sigma}$ and subsequently for $u_{\sigma\sigma\sigma}$. \\
First write \eqref{deqn} as
\begin{equation}\label{apriori1}
V(u) = \Phi_\beta(V(u))+ \kappa(u)-\lambda_u(t)
\end{equation}
and take a derivative in $\sigma$. After rearranging terms:
\begin{equation}\label{vsigma}
V_{\sigma}  = \frac{1}{1-\Phi_\beta'(V)}\kappa_\sigma,
\end{equation} 
Thus, taking a $\sigma$ derivative of \eqref{deqn} and using \eqref{vsigma}:
\begin{equation}\label{usigteqn}
u_{\sigma t} = \frac{S(u)}{1-u\kappa_0}\left(\frac{1}{1-\Phi_\beta'(V(u))} \kappa_\sigma - \left(\frac{1-u\kappa_0}{S(u)}\right)_\sigma u_t\right).
\end{equation}
Taking two $\sigma$ derivatives of \eqref{apriori1} (recalling the definition \eqref{veqn}):
\begin{align}\label{usigsigeqn}
\left(\frac{1-u\kappa_0}{S(u)}\right)_{\sigma\sigma} u_t &+ 2\left(\frac{1-u\kappa_0}{S(u)}\right)_{\sigma} u_{\sigma t}+ \left(\frac{1-u\kappa_0}{S(u)}\right) u_{\sigma\sigma t} \\
&= \Phi_\beta''(V(u))(V(u)_\sigma)^2+\Phi_\beta'(V(u)) V_{\sigma\sigma}+\kappa_{\sigma\sigma}.
\end{align}
Plug  \eqref{deqn}, \eqref{vsigma} and \eqref{usigteqn} into \eqref{usigsigeqn} to write
\begin{equation}\label{usigsigeqn1}
 u_{\sigma\sigma t} = \frac{u_{\sigma\sigma\sigma\sigma}}{S^2(u)(1-\Phi_\beta'(V(u)))} + f,
\end{equation}
where $f=f(u,u_\sigma,u_{\sigma\sigma},u_{\sigma\sigma\sigma},\Phi_\beta(V),\Phi_\beta'(V),\Phi_\beta''(V))$. Due to the smoothness of the initial conditions $u_0$, solutions of \eqref{usigsigeqn1} exist. Indeed, consider \eqref{usigsigeqn1} as a semilinear parabolic equation for $u_{\sigma\sigma}$:
\begin{equation}\label{higherorder_pde}
u_{\sigma\sigma t} = \frac{u_{\sigma\sigma\sigma\sigma}}{S^2(\tilde{u})(1-\Phi_\beta'(V(\tilde{u})))} + f(\tilde{u},\tilde{u}_\sigma,\tilde{u}_{\sigma\sigma},u_{\sigma\sigma\sigma},\Phi_\beta(V(\tilde{u}),\dots)),
\end{equation} 
where $\tilde{u}=\tilde{u}(x,t)$ is the solution of \eqref{deqn}. Solutions of \eqref{higherorder_pde} are guaranteed by classical theory for parabolic PDEs, e.g., \cite{Lad68}.  The compatibility condition that $ u_{\sigma\sigma} = \tilde{u}_{\sigma\sigma}$ follows from the formulation.

We note that $f$ has the form:
\begin{align}
f(u,u_\sigma,u_{\sigma\sigma},u_{\sigma\sigma\sigma}) &= f_1+f_2 u_{\sigma\sigma}+f_3 u_{\sigma\sigma}^2\\
&+f_4u_{\sigma\sigma}^3+f_5 u_{\sigma\sigma}^4+u_{\sigma\sigma\sigma}f_6(u,u_\sigma,u_{\sigma\sigma},u_{\sigma\sigma\sigma}),
\end{align}
where the $f_i=f_i(u,u_\sigma,\Phi_\beta,\Phi_\beta',\Phi_\beta'')$ are uniformly bounded by positive constants, $C_i>\|f_i\|_{L^\infty}$.\newline

Let $b_\ve(t)$ satisfy
\begin{equation}\label{beqn}
\dot{b}_\ve = P_1 +P_2b_\ve +P_3b_\ve^2+P_4b_\ve^3+P_5b_\ve^4
\end{equation}
with $P_i\geq C_i$ fixed constants independent of $u_0$ and $\ve$, with initial condition $b_\ve(0)=\|u_{\sigma\sigma}(0)\|_{L^\infty(I)}+\varepsilon$. Define $T_\ve^* := \sup\{t\geq 0 \colon b_\ve(t)<\infty \}$ and 
\begin{equation}
t_0 = \sup\{t\geq 0; u_{\sigma\sigma}(\sigma,\tau)\leq b_\ve(\tau),\, \forall 0\leq\tau\leq t, \forall\sigma\in I\},
\end{equation} 
the maximal time such that $u_{\sigma\sigma}(\sigma,\tau)\leq b_\ve(\tau)$ for all $\sigma$. Note that $t_0>0$ for all $\ve>0$. \newline

If $t_0\leq\min\{T^*_\ve,T\}$ then choose $\sigma_0$ so that $u_{\sigma\sigma}(\sigma_0,t_0)=\|u_{\sigma\sigma}(t_0)\|_{L^\infty(I)}$. At the point $(\sigma_0,t_0)$ we have $u_{\sigma\sigma\sigma\sigma}(\sigma_0,t_0)\leq 0$, and $u_{\sigma\sigma\sigma}(\sigma_0,t_0)=0$. Moreover, $u_{\sigma\sigma t}(\sigma_0,t_0)\geq \dot{b}_\ve(t_0)$ and $u_{\sigma\sigma}(\sigma_0,t_0)=b_\ve(t_0)$. Then 
\begin{align}\label{max_bound}
\dot{b} &\leq u_{\sigma\sigma t}  = \frac{u_{\sigma\sigma\sigma\sigma}}{S^2(u)(1-\Phi_\beta'(V(u)))}+f \\ 
&< C_1+C_2 u_{\sigma\sigma}+C_3 u_{\sigma\sigma}^2+C_4u_{\sigma\sigma}^3+C_5 u_{\sigma\sigma}^4.
\end{align}
However, \eqref{max_bound} contradicts \eqref{beqn}, thus $t_0> \min\{T_\ve^*,T\}$. Taking the limit $\ve\to 0$ we recover
\begin{equation}
u_{\sigma\sigma}(\sigma,t)\leq b_0(t)
\end{equation}
for all $\sigma\in I$ and all $t\leq T_0^*$. The reverse inequality follows similarly, and thus
\begin{equation}
\|u_{\sigma\sigma}(t)\|_{L^\infty(I)}\leq b(t)
\end{equation}
for all $t<T_0^*$. Note that this result relies on continuity of $t\mapsto \|u_{\sigma\sigma}\|_{L^\infty}$; this property holds for smooth initial data and smooth approximations of coefficient functions. Relaxing to the general case is done in the usual way by use of an approximating sequence.  \newline
Using a bootstrapping argument, we can prove a similar maximum principle to establish uniform estimates on $u_{\sigma\sigma\sigma}$ on some (possibly shorter) time interval $[0,T^{**}]$; we omit the details here.
\end{proof}


Although solutions of \eqref{deqn} do not necessarily satisfy a Poincar\'e inequality we prove that the difference between solutions does.


\begin{lemma}\label{lem_poincare}
Let $u$ and $v$ be solutions of \eqref{deqn} with the same initial condition $u_0\in W^{1,\infty}(I)$ on some shared interval $[0,T]$. Then, their difference $w:=u-v$ satisfies
\begin{equation}
\|w(t)\|_{L^2(I)}\leq C \|w_\sigma(t)\|_{L^2(I)}
\end{equation}
for all $t\in [0,T]$.
\end{lemma}
\begin{proof}
We prove that for each time $t\leq T$ there is a point $\sigma_0$ where the curves $u$ and $v$ intersect. To that end, suppose to the contrary that there exists some moment of time $t$ such that $u(\sigma,t)>v(\sigma,t)$ for all $\sigma\in I$. Let $R(u)$ be the region enclosed by the curve defined by $u$ (see equation \eqref{curvetou}) and likewise for $R(v)$. By continuity of $u$ and $v$ and the assumption that $u>v$, there exists a ball $B$ with sufficiently small radius so that $B \subset R(v)\setminus R(u)$.  Then,
\begin{equation}
\mu\{R(u)\}< \mu\{R(u)\}+\mu\{B\} \leq \mu\{R(v)\},
\end{equation}
where $\mu$ is the 2 dimensional Lebesgue measure. This contradicts preservation of volume. Thus for each moment of time, there exists a point $\sigma_0(t)\in I$ such that $u(\sigma_0(t),t)=v(\sigma_0(t),t)$. We can thus write
\begin{equation}
|w(\sigma,t)| \leq \int_{\sigma_0(t)}^\sigma |w_\sigma(s)|ds \leq |\sigma-\sigma_0(t)|^{1/2} \|w_\sigma\|_{L^2(I)}\leq |I|^{1/2}\|w_\sigma(t)\|_{L^2(I)}
\end{equation}
Squaring and integrating over $I$ yields the desired inequality. \end{proof}

 We are now prepared to prove Theorem \ref{thm_unique}. A natural approach is to assume the existence of two solutions $u$ and $v$ and to use equation \eqref{deqn} to write an equation for $w:=u-v$. Multiplying by an appropriate factor (e.g., $w_t$) and integrating over the space domain one expects that both principal terms containing $w_t$ (i.e., local term and non-local term) to have appropriate signs in order to establish a Gr\"onwall inequality for the growth of $\|w_\sigma(t)\|_{L^2(I)}$. 
 
 However, due to non-linearity, we can only guarantee that one of the two principal terms (say, the local term) containing $w_t$ appears in the Gr\"onwall estimate with correct sign. Thus, the main difficulty of the proof is to estimate the non-local term which contains a non-linear function of $w_t$; the key to resolve this difficulty is a manipulation of the volume constraining condition: as a result one can exchange non-local terms containing $w_t$ for non-local terms containing $w$. This manipulation relies on special choice of multiplication factor.

\begin{proof}[Proof of Theorem \ref{thm_unique}]
Suppose there exist two solutions of \eqref{deqn}, $u,v\in L^2(0,T;H^2(I))$ with $u_t,v_t\in L^2(0,T;L^2(I))$ satisfying the same initial condition $u_0\in W^{3,\infty}(I)$ with  $\|u_0\|_{L^\infty}<\delta_0$. By Lemma \ref{lem_usigsig_est} we have $u,v\in L^\infty(0,T; W^{3,\infty}(I))$ (on some possibly shorter time interval, which we denote with the same $T$). Then
\begin{align}
\frac{1-u\kappa_0}{S(u)}u_t-\frac{1-v\kappa_0}{S(v)}v_t &=\Phi_\beta\left(\frac{1-u\kappa_0}{S(u)}u_t\right)-\Phi_\beta\left(\frac{1-v\kappa_0}{S(v)}v_t\right)\\
&+ \frac{1-u\kappa_0}{S^3(u)}u_{\sigma\sigma}-\frac{1-v\kappa_0}{S^3(v)} u_{\sigma\sigma}\\ &+f(u,u_\sigma)-f(v,v_\sigma) -\lambda_u + \lambda_v,
\end{align}
for some function $f$. There exists some $z\in L^2(I\times [0,T])$ so that
\begin{align}
(1-\Phi_\beta'(z))\left(\frac{1-u\kappa_0}{S(u)}u_t-\frac{1-v\kappa_0}{S(v)}v_t\right)&=\frac{1-u\kappa_0}{S^3(u)}u_{\sigma\sigma}-\frac{1-v\kappa_0}{S^3(v)} v_{\sigma\sigma} \\
&+f(u,u_\sigma)-f(v,v_\sigma) -\lambda_u + \lambda_v.
\end{align}
Let $w:=u-v$. Upon rearrangement of terms:
\begin{align}\label{difference_eqn}
\frac{(1-\Phi_\beta'(z))(1-u\kappa_0)}{S(u)}w_t-\frac{1-u\kappa_0}{S^3(u)}w_{\sigma\sigma} = f_1 w + f_2 w_\sigma -\lambda_u +\lambda_v,
\end{align}
where $f_i=f_i(u,u_\sigma,u_{\sigma\sigma},u_t,v,v_\sigma,v_{\sigma\sigma},v_t,\Phi_\beta(V),\kappa_0)$ are uniformly bounded by virtue of Lemma \ref{lem_usigsig_est}. Multiply \eqref{difference_eqn} by $(1-u\kappa_0)w_t$ and integrate over $I$:
\begin{align}\label{proof_ineq}
\int_I \frac{(1-\Phi_\beta'(z))(1-u\kappa_0)^2}{S(u)}w_t^2 d\sigma &-\int_I \frac{(1-u\kappa_0)^2}{S^3(u)} w_{\sigma\sigma} w_t d\sigma = \int_I f_1 w w_t d\sigma\\
&+\int_I f_2 w_\sigma w_t d\sigma -(\lambda_u-\lambda_v)\int(1-u\kappa_0)w_t d\sigma,
\end{align}
where $f_1$ and $f_2$ have been updated accordingly. Note that since $\|\Phi_\beta'\|_{L^\infty(\mathbb{R})}<1$ then the first term on the left hand side of \eqref{proof_ineq} is strictly positive.\\
Integrating by parts
\begin{equation}
-\int_I \frac{(1-u\kappa_0)^2}{S^3(u)} w_{\sigma\sigma}w_t d\sigma  = \int_I \left(\frac{(1-u\kappa_0)^2}{S^3(u)}\right)_\sigma w_\sigma w_t d\sigma  + \int_I  \frac{(1-u\kappa_0)^2}{S^3(u)} w_\sigma w_{\sigma t}
\end{equation}
We note that
\begin{equation}
\int_I \frac{(1-u\kappa_0)^2}{S^3(u)} w_\sigma w_{\sigma t} = \frac{1}{2}\frac{d}{dt} \int_I \frac{(1-u\kappa_0)^2}{S^3(u)} w_\sigma^2 - \frac{1}{2} \int_I \left(\frac{(1-u\kappa_0)^2}{S^3(u)}\right)_t w_\sigma^2
\end{equation}
Note that the term $((1-u\kappa_0)^2S^{-3}(u))_t$ contains terms of the form $u_t, u_{\sigma t}$, which are uniformly bounded by Lemma \ref{lem_usigsig_est}. Using Young's inequality on the first two terms on the right hand side we then establish
\begin{align}\label{almost}
c\int_I w_t^2 d\sigma + \frac{1}{2}\frac{d}{dt}\int_I \frac{(1-u\kappa_0)^2}{S^3(u)}w_\sigma^2 d\sigma &\leq C\left(\int_I w^2 d\sigma+\int_I w_\sigma^2 d\sigma\right)\\
&-(\lambda_u-\lambda_v)\int(1-u\kappa_0)w_td\sigma,
\end{align}
for some $c,C>0$. \newline

We require an estimate on the last term on the right hand side of \eqref{almost}. To that end first multiply \eqref{deqn1} by $S(u)$ and integrate over $I$, using the definition of $\lambda_u$ and the fact that $\frac{1}{L[u]}\int_I \kappa(u) S(u)d\sigma = \frac{1}{|\Gamma|} \int_{\Gamma}\kappa ds=2\pi$ (see, e.g., \cite{Doc76}) to deduce that $\int_I V(u) S(u) d\sigma = 0$. That is, 
\begin{equation}
\int_I(1-u\kappa_0)u_t d\sigma=0,
\end{equation} 
and similarly for $v$. Thus, it follows that
\begin{equation}\label{wtforw}
\int_I (1-u\kappa_0) w_t d\sigma = \int_I \kappa_0 v_t w d\sigma
\end{equation}
and so we have
\begin{equation}
(\lambda_u-\lambda_v)\int(1-u\kappa_0)w_td\sigma \leq \xi |\lambda_u-\lambda_v|^2+C_\xi \int_I w^2d\sigma,
\end{equation}
for any $\xi>0$.
However, immediate calculations show that
\begin{equation}\label{lambdaest}
|\lambda_u-\lambda_v|^2 \leq C_1\int_I w^2 d\sigma + C_2 \int_I w_\sigma^2 d\sigma + C_3\int_I w_t^2 d\sigma,
\end{equation}
where constants $C_i$ can be chosen independent of $u$ and $v$.

Combining \eqref{almost} with \eqref{lambdaest}, we conclude that for $\xi>0$ sufficiently small:
\begin{equation}
\frac{d}{dt} \left\|\frac{(1-u\kappa_0)}{S^{3/2}(u)}w_\sigma\right\|^2_{L^2(I)} \leq C_\xi \left(\|w\|^2_{L^2(I)}+\|w_\sigma\|^2_{L^2(I)}\right)
\end{equation}
By Lemma \ref{lem_poincare}, $w$ satisfies the Poincar\'{e} inequality and so
\begin{equation}
\frac{d}{dt} \left\|\frac{(1-u\kappa_0)}{S^{3/2}(u)}w_\sigma\right\|^2_{L^2(I)}\leq C_\xi \|w_\sigma\|^2_{L^2(I)} \leq \tilde{C}_\xi \left\|\frac{(1-u\kappa_0)}{S^{3/2}(u)}w_\sigma\right\|^2_{L^2(I)}
\end{equation}
Since $w_\sigma(\sigma,0)\equiv 0$ and $\frac{(1-u\kappa_0)}{S^{3/2}(u)}>0$ we apply the Gr\"{o}nwall inequality to conclude that $w_\sigma\equiv 0$. \newline

Thus, $u$ and $v$ differ only by a function which is constant in space: $u(\sigma, t) = v(\sigma,t)+g(t)$. However due the volume preservation constraint and the fact that $g(0)=0$, it follows that $g\equiv 0$ and uniqueness is proved. \end{proof}

\section{2D traveling wave solutions}\label{sec_nonexist}

A traveling wave solution is a family of smooth curves which translate with constant shape and velocity. That is:
\begin{definition}
A {\bf traveling wave solution} of \eqref{interface} is a smooth ($C^2$) family of curves $\Gamma(\sigma,t)$ evolving by \eqref{interface} which satisfies
\begin{equation}\label{traveling_wave}
\Gamma(\sigma,t) = \Gamma(\sigma,0)+V_0 t,
\end{equation}
for some initial curve $\Gamma(\sigma,0)$ (the traveling wave profile) and $V_0\in \mathbb{R}^2$. 
\end{definition}

From the point of view of cell motility, traveling wave solutions correspond to persistently moving cells.  We note that stationary ($V=0$) circular solutions of \eqref{interface} always trivially exist. These solutions correspond to non-motile cells. \\

In \cite{MizBerRybZha15} it was proved that in the subcritical regime there are no traveling wave solutions of \eqref{interface} other than stationary circles. Since the parameter $\beta$ is related to the internal biophysics, this was interpreted as a regime where the physical mechanisms of actin polymerization and adhesion were too weak to overcome the membrane tension and give rise to persistent motion. The proof relied on the fact that $V-\Phi_\beta(V)$ is a monotone increasing function of $V$. \newline

In this section we will prove a result for non-existence of (non-trivial) traveling wave solutions as well as a sufficient condition for the existence of traveling wave solutions.

\subsection{Non-existence of traveling waves for symmetric $\Phi_\beta$} 
We now prove that if $\Phi_\beta$ is symmetric over the $y$-axis ($\Phi_\beta(V)=\Phi_\beta(-V)$), then there are no non-trivial traveling wave solutions of \eqref{interface}. In particular there are no traveling wave solutions in the case that $W$ is symmetric even when $\beta>\beta_{crit}$. This result was established in \cite{BerPotRyb16} for the 1D case; however the 2D case is not immediately clear since there is additional (geometric) freedom in 2D compared to 1D. 
This theorem extends the results of \cite{MizBerRybZha15} to the supercritical $\beta$ regime; symmetry of $\Phi_\beta$ replaces the role that monotonicity of $V-\Phi_\beta(V)$ had in the subcritical $\beta$ regime. 

\begin{theorem}\label{thm_nonexist}
Let $\Phi_\beta\in W^{1,\infty}(\mathbb{R})$ be symmetric: $\Phi_\beta(V)=\Phi_\beta(-V)$ for all $V\in\mathbb{R}$. If $\Gamma(\sigma,t)$ is a traveling wave solution of \eqref{interface} then $V_0=0$ and $\Gamma(\sigma,0)$ is a circle.
\end{theorem}

\begin{proof}
Assume there exists a non-trivial traveling wave solution with non-zero $V_0\in \mathbb{R}^2$. By rotation and translation we may assume that $V_0 = (0,v)$ with $v>0$ and that $\Gamma(\sigma,t)$ is contained in the upper half plane for all $t\geq 0$. Let $\sigma_0$ be such that $\Gamma(\sigma_0,0)$ is the closest to the $x$-axis. By translating we may assume that $\Gamma(\sigma_0,0)=(0,0)$. Locally we represent $\Gamma(\sigma,t)$ as a graph over the $x$-axis, $y=y(x)+ct$, where $y$ solves
\begin{equation}\label{y_eqn}
y'' = f_\lambda^v(y')
\end{equation}
with $y(0)=y'(0)=0$ and
\begin{equation}\label{f_eqn}
f_\lambda^v(z) := (1+z^2)^{3/2}\left(\frac{v}{\sqrt{1+z^2}}-\Phi_\beta\left(\frac{v}{\sqrt{1+z^2}}\right)+\lambda\right).
\end{equation}
Setting $w:=y'$ we see that solving \eqref{y_eqn} is equivalent to the first order equation
\begin{equation}\label{w_eqn}
w' = f_\lambda^v(w)
\end{equation}
with $w(0)=0$. If $w$ has a global solution then it cannot describe part of a smooth closed curve, so we may assume that the parameters $\lambda$ and $v$ are such that the solution $w_B(x)$ has finite blow-up $w_B\to \infty$ as $x\to x_B^*>0$. Note that \eqref{w_eqn} is uniquely solvable on its interval of existence by Lipschitz continuity and moreover exhibits symmetry over the $y$-axis by definition of $f_\lambda^v$. In particular the interval of existence of $w_B$ is $(-x_B^*,x_B^*)$. Defining $y_B(x):=\int_0^x w_B(s)ds$ we see that $y_B$ has a vertical tangent at $x_B^*$.\\

 The subscript $B$ suggests that $y_B$ represents the ``back'' portion of the curve. To form the front of the curve we consider $w_F$ (where $F$ stands for the ``front'' portion of the curve), where $w_F$ solves \eqref{w_eqn} with right hand side $f_\lambda^{-v}$ and initial condition $w_F(0)=0$. As above we assume that $w_F$ has a blow-up at $0<x_F^*<\infty$. Defining $y_F(x):=\int_0^x w_F(s)ds$ we have the transformation
 \begin{equation}
 \hat{y}_F(x) := -y_F(x-(x_B^*-x_F^*))+y_B(x^*_B)+y_B(x_F^*),
 \end{equation}
 which permits a smooth gluing of $\hat{y}_F$ to $y_B$ at the point $(x_B^*,y_B(x_B^*))$. It was proved in \cite{MizBerRybZha15} that this is the unique, smooth extension of $y_B$ at $x_B^*$; we omit the proof here. \newline
 
 We will prove that $x^*_F>x_B^*$, which guarantees that the graphs of $y_B(x)$ and $\hat{y}_F(x)$ cannot smoothly meet at $-x_B^*$. To that end, we first note that $w_B(0)=w_F(0)=0$ implies that
 \begin{equation}
 w_B'(0)-w_F'(0) = v-\Phi_\beta(v)-(-v-\Phi_\beta(-v)) = 2v>0,
 \end{equation}
 by the symmetry of $\Phi_\beta$. Thus, $w_B'(0)>w_F'(0)$. By continuity of $w'_B$ and $w'_F$ we deduce that $w_B(x)>w_F(x)$ for all $x>0$ sufficiently small. Suppose that there exists $\tilde{x}<\min\{x_F^*,x_B^*\}$ such that $w_B(\tilde{x})=w_F(\tilde{x})$. Necessarily at this point $w_F'(\tilde{x})\geq w_B'(\tilde{x})$. However, using \eqref{w_eqn} (and the symmetry of $\Phi_\beta$) we deduce that at $\tilde{x}$:
 \begin{equation}
 0\leq w_F'(\tilde{x})-w_B'(\tilde{x})
 = -2v(1+w_B(\tilde{x})^2),
 \end{equation}
 implying that $v\leq 0$, a contradiction. Thus $w_B(x)>w_F(x)$ for all $x < \min\{x^*_F,x_B^*\}$ and so $x^*_F\geq x^*_B$. \\
 
 Take $\hat{x}_F>0$ and $\hat{x}_B>0$ be such that $w_F(\hat{x}_F)=w_B(\hat{x}_B)$. Clearly $\hat{x}_F>\hat{x}_B$. Since
 \begin{equation}\label{growth_ineq}
 w_B'(\hat{x}_B)-w_F'(\hat{x}_F)=2v(1+w_B(\hat{x}_B)^2)>0,
 \end{equation}
 then $w_B'(\hat{x}_B)>w_F'(\hat{x}_F)$.
 Consider $w_{new}$ which solves
 \begin{equation}
 w_{new}' = f_\lambda^v(w_{new})
 \end{equation}
 with initial condition $w_{new}(\hat{x}_F) = w_F(\hat{x}_F)$ (see Figure \ref{nonexist_proof} for a sketch of $w_B$, $w_F$, and $w_{new}$). From \eqref{growth_ineq} and the initial condition of $w_{new}$ we deduce that $w'_{new}(\hat{x}_F)>w'_F(\hat{x}_F)$. Moreover, repeating the estimates as above we see that $w_{new}(x)> w_F(x)$ for all $x>\hat{x}_F$. Since 
 \begin{equation} 
 w_{new}(x) = w_B(x-(\hat{x}_F-\hat{x}_B))
 \end{equation}
 then we conclude that
 \begin{equation}
 x_B^* \leq x_F^*-(\hat{x}_F-\hat{x}_B) < x_F^*
 \end{equation}
 completing the proof.
\end{proof}

\begin{figure}[h!]
\centering
\includegraphics[width = .5\textwidth]{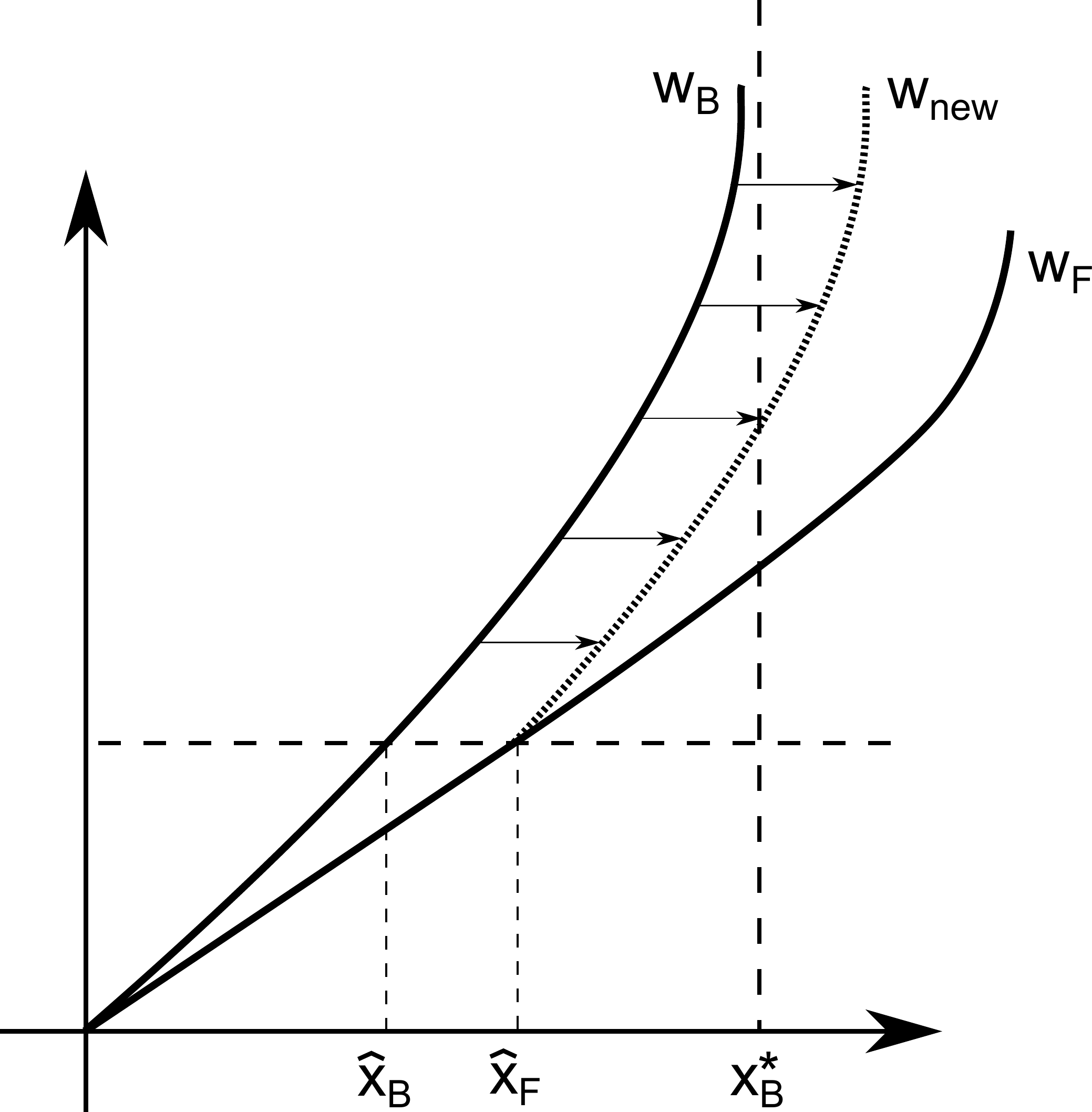}
\label{nonexist_proof}
\caption{Sketch of functions $w_B$, $w_F$ and $w_{new}$ from the proof of Theorem \ref{thm_nonexist}}
\end{figure}

\subsection{Existence of traveling waves for asymmetric $\Phi_\beta$}
As motile cells exhibit persistent motion, it is desirable to capture non-trivial traveling wave solutions.
We prove the existence of non-trivial traveling wave solutions to \eqref{interface}
in the case that $\Phi_\beta(V)$ possesses a sufficient level of asymmetry, namely that $\Phi_\beta'(0)>1$.

\begin{remark}
The condition $\Phi_\beta'(0)>1$ guarantees that 
\begin{equation}\label{tw_1D}
2V = \Phi_\beta(V)-\Phi_\beta(-V)
\end{equation}
 has a positive solution $V$. Solvability of \eqref{tw_1D} is both necessary and sufficient for the existence of traveling waves in the 1D case. This is not the case in 2D, where solvability of \eqref{tw_1D} is necessary but not sufficient.
\end{remark}

As in the previous subsection, we construct graphs of the form $y(x)+Vt$ where $y$ solves
\begin{equation}\label{yb_eqn}
y_B'' = f_V^\lambda (y_B'),\quad y_B(0)=y_B'(0)=0
\end{equation}
with $f_V^\lambda$ as defined in \eqref{f_eqn}. Then, $y$ corresponds to a (non-closed) curve which evolves with constant velocity $V$.
Equation \eqref{yb_eqn} has a maximal interval of existence $I_B$.  Likewise we consider
\begin{equation}\label{yf_eqn}
y_F''=f_{-V}^\lambda(y_F'), \quad y_F(0)=y'_F(0)=0.
\end{equation}
Equation \eqref{yf_eqn} has a maximal interval of existence $I_F$. The graphs $y_B$ and $y_F$ represent the rear and front parts of the traveling wave curve, respectively. 
Here, $V$ and $\lambda$ are parameters which must be chosen appropriately so as to ensure the intervals of existence, $I_B=(-x_B,x_B)$ and $I_F=(-x_F,x_F)$, are equal. In the case that $I_B=I_F$, it is clear that $\tilde{y}_F(x)= -y_F(x)+y_B(x_B)$ and the subsequent gluing of $\tilde{y}_F$ to $y_B$ gives rise to a smooth curve which is a traveling wave profile. \newline
Letting $w_B=y_B'$ we have
\begin{equation}
w_B' = f_V^\lambda(w_B),\quad w_B(0)=0,
\end{equation}
and likewise
\begin{equation}
w_F' = f_{-V}^\lambda (w_F),\quad w_F(0)=0.
\end{equation}
Define
\begin{equation}
\lambda(V) := 2\|\Phi_\beta\|_{L^\infty(\mathbb{R})}+V.
\end{equation}
Then $f^{\lambda(V)}_V(z)>0$ and likewise $f^{\lambda(V)}_{-V}(z)>0$ for all $V$ and all $z$. We then consider
\begin{equation}\label{wb_eqn}
w_B' = f_V^{\lambda(V)}(w_B),\quad w_B(0)=0,
\end{equation}
and likewise
\begin{equation}\label{wf_eqn}
w_F' = f_{-V}^{\lambda(V)} (w_F),\quad w_F(0)=0.
\end{equation}
We first prove the following lemma which states that $I_B$ and $I_F$ are finite intervals, ensuring the finite blow-up of $w_B$ and $w_F$.
\begin{lemma}\label{lem_blowup}
The maximal intervals of existence $I_B$ and $I_F$, corresponding to the ODEs \eqref{wb_eqn}-\eqref{wf_eqn}, are finite for any value of $V$.
\end{lemma}

\begin{proof}
Define 
\begin{equation}
c(V) := \min_{0\leq\tilde{V}\leq V} \{\tilde{V}-\Phi_\beta(\tilde{V})+\lambda(V)\}
\end{equation}
and
\begin{equation}
d(V) := \min_{0\leq\tilde{V}\leq V} \{-\tilde{V}-\Phi_\beta(-\tilde{V})+\lambda(V)\}
\end{equation}
By definition of $\lambda(V)$, it follows that $c(V)\geq \|\Phi_\beta\|_{L^\infty(\mathbb{R})}>0$ and $d(V)\geq \|\Phi_\beta\|_{L^\infty(\mathbb{R})}>0$ for any $V$. Thus, we may bound the growth of $w_B$ and $w_F$ below:
\begin{equation}
w_B' \geq c(V) (1+w_B^2)^{3/2} 
\end{equation}
and likewise
\begin{equation}
w_F' \geq d(V) (1+w_F^2)^{3/2}.
\end{equation}
Let $A=\min\{c(V),d(V)\}$ and consider the ODE
\begin{equation}
v' = A (1+v^2)^{3/2}, \quad v(0)=0.
\end{equation}
Its solution
\begin{equation}
v(x) = \frac{Ax}{\sqrt{1-A^2x^2}}
\end{equation}
has blow-up $\lim_{x\to 1/A^-} v(x) = +\infty$. Due to the point-wise estimate $w_B\geq v$ and $w_F\geq v$ for all $x$, we have established the blow-up of $w_B$ and $w_F$.
\end{proof}

We have the following identities for the times of blow-up of $w_B$ and $w_F$:
\begin{equation}\label{blowup_b}
x_B = \int_0^{\infty} \frac{dz}{f_V^{\lambda(V)}(z)}
\end{equation}
and
\begin{equation}\label{blowup_f}
x_F = \int_0^{\infty} \frac{dz}{f_{-V}^{\lambda(V)}(z)}.
\end{equation}

By Lemma \ref{lem_blowup}, $x_B$ and $x_F$ are finite, so we may define
\begin{align}\label{integral}
I(V) &:= \int_0^\infty \left(\frac{1}{f^{\lambda(V)}_{-V}(z)}-\frac{1}{f^{\lambda(V)}_{V}(z)} \right)dz
\end{align}
Making the substitution $z\mapsto V/(1+z^2)^{1/2}$ in \eqref{integral}:
\begin{equation}\label{newintegral}
I(V) = \frac{1}{V} \int_{0}^{V} \frac{z(2z+\Phi_\beta(-z)-\Phi_\beta(z))}{\sqrt{V^2-z^2}(z-\Phi_\beta(z)+\lambda(V))(-z-\Phi_\beta(-z)+\lambda(V))}dz
\end{equation}
It is clear that zeros of the function $I$ correspond precisely to smooth traveling wave solutions.
\begin{theorem}\label{thm_twexist}
Let $\Phi_\beta$ satisfy $\Phi_\beta'(0)>1$. Then there exists a non-zero velocity traveling wave solution to \eqref{interface}.
\end{theorem}

\begin{proof}
	 Choose  $V^*>0$  such that  $ V^* > 2 \|\Phi_\beta\|_{L^\infty(\mathbb{R})} $. 
	The condition that $\Phi_\beta'(0)>1$ implies that $2V< \Phi_\beta(V)-\Phi_\beta(-V)$ for all $V$ which are sufficiently small. We conclude that $I(V)<0$ for $V$ sufficiently small. We will prove $ I(V) >0 $ for $V$ sufficiently large. Thus, by continuity of $I$ there exists a positive velocity $\bar{V}$ which is a zero of $I$, corresponding to a traveling wave velocity. Assume now that $V>V^*$. Then,
	\begin{equation}
		VI(V) = \int_0^{V^*} \cdot dz + \int_{V^*}^V \cdot dz = I_1(V)+I_2(V).
	\end{equation}
	Let $K:= \max_{0\leq z\leq V^*}\{ | z(2z+\Phi_\beta(-z)-\Phi_\beta(z)) |\}$. Then,
	\begin{equation}
	|	I_1(V) | \leq \frac{K}{\|\Phi_\beta\|^2_{L^\infty(\mathbb{R})}} \int_0^{V^*} \frac{1}{\sqrt{V^2-z^2}} = \frac{K}{\|\Phi_\beta\|^2_{L^\infty(\mathbb{R})}}\arctan\left(\frac{V^*}{\sqrt{V^2-(V^*)^2}}\right)
	\end{equation}
	Note that $I_1(V) \to 0$ as $V\to \infty$. \\
	We now estimate $I_2(V)$ from below:
	\begin{align*}
		 I_2(V) &\geq \int_{V^*}^{V} \frac{2z^2-z(\Phi_\beta(z)-\Phi_\beta(-z))}{V(2V+2\|\Phi_\beta\|_{L^\infty(\mathbb{R})})^2} \\
		& \geq \int_{V^*}^{V} \frac{z^2}{V(2V+2\|\Phi_\beta\|_{L^\infty(\mathbb{R})})^2} \\
		&= \frac{1}{3V(2V+2\|\Phi_\beta\|_{L^\infty(\mathbb{R})})^2} (V^3-(V^*)^3) \\		
	\end{align*}
	Upon taking the limit, $V\to\infty$ we see that
	\begin{equation}
		\lim_{V\to \infty} I_2(V) \geq \frac{1}{12}.
	\end{equation}
	We conclude that for $V$ sufficiently large, we have $I(V)>0$, completing the proof.
	 \end{proof}

\begin{remark}
For example, in \cite{BerPotRyb16}, an asymmetric double-well potential $W(z) = \frac{1}{4}z^2(1-z)^2(1+z^2)$ is considered in the phase-field system \eqref{eq1}-\eqref{eq2} and it is seen that in the sharp interface limit, where $\Phi_\beta'(0)>0$.  It follows that if $\beta$ is sufficiently large, then $\beta\Phi_\beta'(0)>1$. Thus, asymmetry in the double-well potential for the phase-field model is sufficient to give rise to persistent motion; physically, we recall that asymmetric potentials can be the result of myosin contraction.
\end{remark}

\section{Numerical results}

\subsection{Constructing traveling waves}

We numerically determine parameters $V$ and $\lambda$ corresponding to traveling wave solutions, i.e., such that solutions of  \eqref{yb_eqn}-\eqref{yf_eqn} have the same interval of existence. In the spirit of the proof of Theorem \ref{thm_twexist} we may define the integral
\begin{equation}\label{new_integral1}
I(V,\lambda) = \int_{0}^{\infty} \left( \frac{1}{f^\lambda_{-V}(z)}-\frac{1}{f^\lambda_V(z)}\right) dz
\end{equation}
whose zeros correspond to parameters yielding traveling wave solutions.  However, due to the unbounded domain and singularities of the integrand, solving \eqref{new_integral1} may accumulate significant numerical error. Thus, we introduce the following algorithm to search for traveling wave solutions:\\
\begin{algorithm}
\begin{enumerate}
\item[(I)] Solve $y_B'' = f_V^\lambda(y_B')$, with $y_B(0)=y_B'(0)=0$ until $y_B'(x_B)\approx 1$. \\
\item [(II)] Rotate the plane clockwise by $\frac{\pi}{2}$. In this frame, the traveling wave moves with velocity $v_x=V$, $v_y=0$, and can be locally represented as the graph
$x=x(y)+Vt$,
 $x(y)$ solving
\begin{equation}\label{rotatedeqn}
x'' = g_V^\lambda(x')
\end{equation}
with
\begin{equation}
g_V^\lambda(z):= \left(\frac{-Vz}{\sqrt{1+z^2}}-\Phi_\beta\left(\frac{-Vz}{\sqrt{1+z^2}}\right)+\lambda\right)(1+z^2)^{3/2}.
\end{equation} 
In this frame, solve $y_R'' = g_V^\lambda(y_R')$ with $y_R(0)=0, y'_R(0)= -y_B'(x_B)$ until $y_R'(x_R)\approx 1$.  \\
\item[(III)] Again rotate the plane clockwise by $\frac{\pi}{2}$. In this frame, the traveling wave moves with velocity $v_x=0$, $v_y=-V$. In this frame solve $y_F'' = f_{-V}^\lambda(y_F')$, with $y_F(0)=0$,  $y_F'(0)=-y_R'(x_R)$. \\
(IV). Define  $I_2(V,\lambda):=y_F'(x_B-y_R(x_R))$. Then $I_2(V,\lambda)=0$ if and only if the pair $(V,\lambda)$ corresponds to a traveling wave solution. 
\end{enumerate}
\end{algorithm}

For steps $(I)-(III)$ we use standard numerical packages to solve the ODEs. By assumption the algorithm avoids blow-up of derivatives and thus standard differential equations solvers are sufficiently accurate for subsequent numerical simulations.

As a toy example we define $\tilde{\Phi}_\beta(V) = -\beta (1-\tanh(V))e^{-V^2}$. We note that this choice of $\tilde{\Phi}_\beta$ is qualitatively similar to the function resulting from the phase-field model with asymmetric potential well. In particular, $\tilde{\Phi}_\beta$ has the property that $\tilde{\Phi}_\beta'(0)>1$ for sufficiently large $\beta$ and has exponential decay at $x\to \pm \infty$. The plot of $I_2(V,\lambda)$ with $\beta=100$ is in Figure \ref{figs:fake_pictures}.


\begin{figure}[h]
\includegraphics[width=\textwidth]{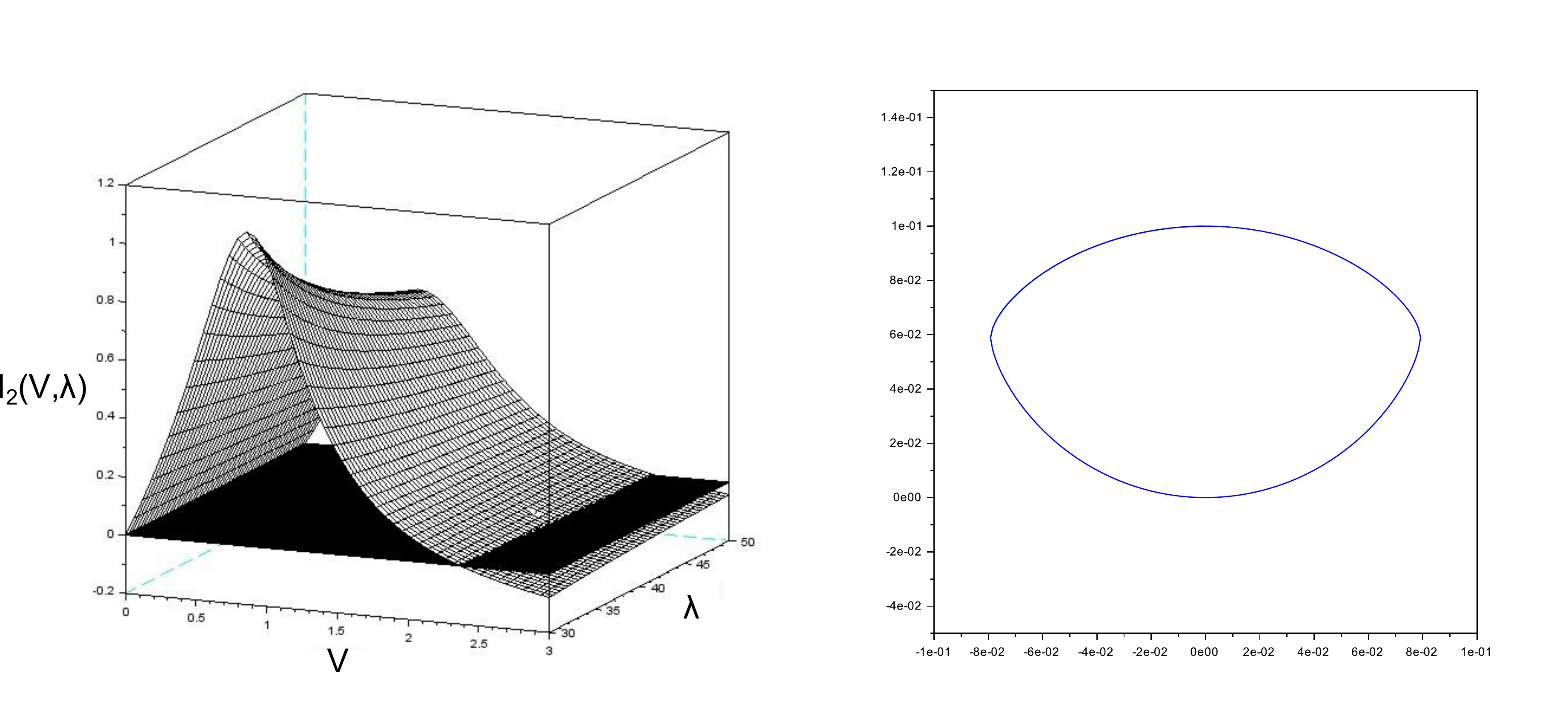}
\caption{(Left) Plot of $I_2(V,\lambda)$ with $\tilde{\Phi}_\beta(V)$ and $\beta=100$ (Right) Traveling wave profile for $\tilde{\Phi}_\beta(V)$, $\beta=100$, $V \approx 2.15$ (in positive $y$ direction), $\lambda \approx 9.75$.}
\label{figs:fake_pictures}
\end{figure}

We take the particular values $V \approx 2.15$, $\lambda \approx 9.75$ and plot the resultant traveling wave profile in Figure \ref{figs:fake_pictures}.
%

In order to relate the traveling wave analysis to the original phase-field model, we take the asymmetric potential well:
\begin{equation}\label{asym_pot_well}
W(z) = \frac{1}{4}z^2(1-z)^2(1+150z^2), 
\end{equation}
and solve for $\Phi_\beta(V)$ using \eqref{standingwave}-\eqref{phi1}. Again, taking $\beta = 100$ we plot $I_2(V,\lambda)$ in Figure \ref{figs:real_profile}.


As expected, we observe that $V=0$ corresponds to a family of non-motile circular solutions. We also observe a distinct family of motile solutions. Taking $V\approx 1.7$, $\lambda \approx 0$ corresponds to the traveling wave solution plotted in Figure \ref{figs:real_profile}.  


\begin{figure}[h]
\includegraphics[width = \textwidth]{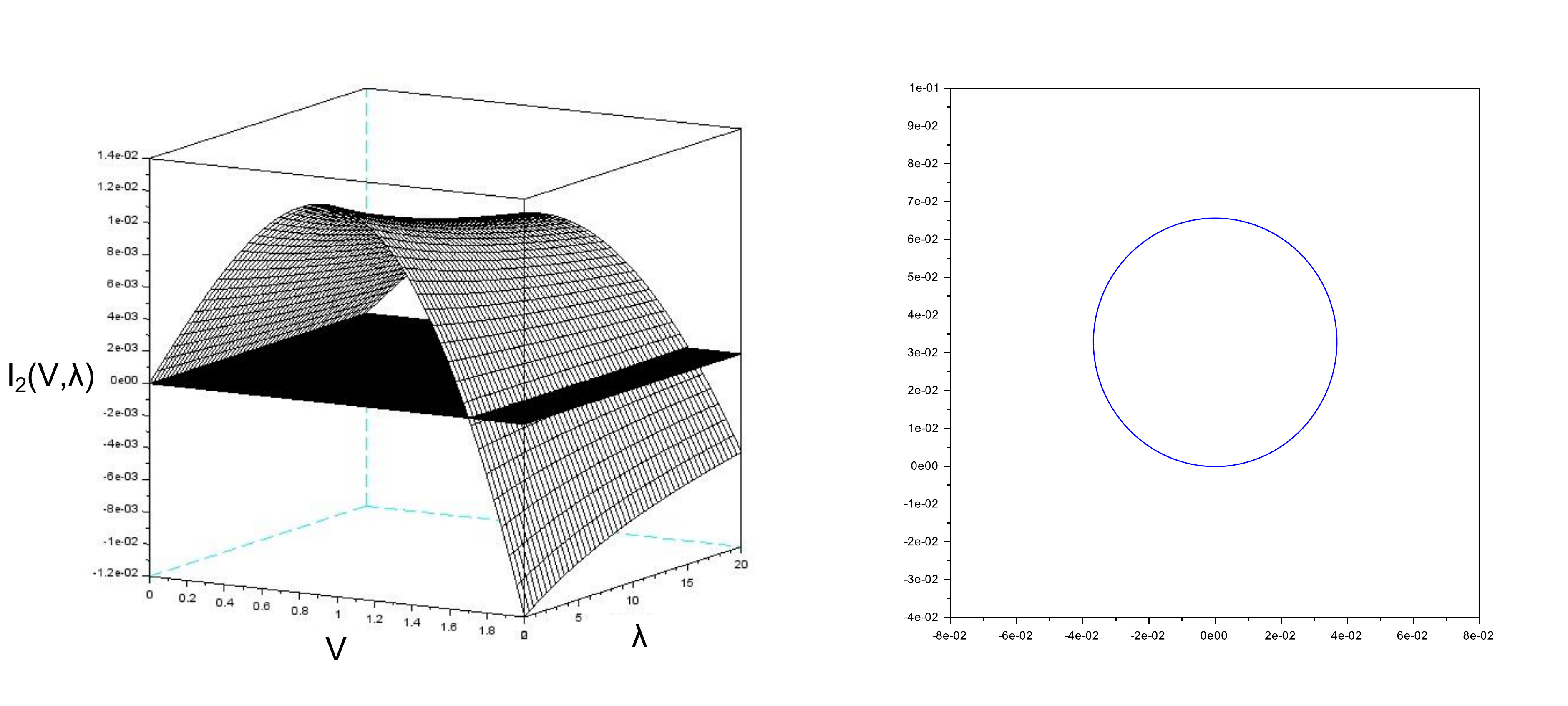}
\caption{(Left) Plot of $I_2(V,\lambda)$ with $\Phi_\beta(V)$ derived from \eqref{standingwave}-\eqref{phi1} with $W$ as in \eqref{asym_pot_well}; $\beta=100$. (Right) Traveling wave profile with $V\approx 1.7$, $\lambda \approx 0$.}
\label{figs:real_profile}
\end{figure}

\subsection{Algorithm for dynamic simulation of the sharp interface limit equation}

In the case $\beta=0$ (corresponding to volume preserving curvature motion), efficient techniques such as level-set methods \cite{Osh88,Sme03} and diffusion generated motion methods \cite{Mer93,Ruu03} can be used to accurately simulate the evolution of curves by \eqref{interface}. Furthermore, other modern numerical schemes (e.g., \cite{Barr11, Dec95, Dec05, Mik04}) are well suited to solve more general geometric flow problems. However, there is no straightforward way to implement these methods when $\beta>0$ since $V$ may not be uniquely defined by \eqref{interface}. 

We propose an algorithm to numerically investigate the dynamics of curves evolving via \eqref{interface} in order to study stability of traveling wave solutions as well as search for new modes of motion. Due to non-uniqueness of solutions of \eqref{interface} we require a mechanism to choose the most ``physically relevant'' solution. Recall the intermediate system \eqref{numeric_veqn1}-\eqref{numeric_aeqn1}:
\begin{align}
V_\epsilon(s,t) &= \kappa_\epsilon(s,t) + \int_{\mathbb{R}} A_\epsilon(s,z,t) (\theta_0'(z))^2 dz - \lambda_\epsilon(t) \label{numeric_veqn} \\
\epsilon \partial_t A_\epsilon(s,z,t) &= \partial_{zz}A_\epsilon + \partial_z A_\epsilon V_\epsilon -A_\epsilon - \beta\theta_0'(z),\;\; A_\epsilon (\pm \infty) = 0. \label{numeric_aeqn}
\end{align}
The intermediate system (between the full phase-field model and the sharp interface limit) is uniquely solvable and can be used as an evolution equation for planar curves.  Moreover, (formally) taking the limit $\ve \to 0$ yields \eqref{interface}.


Note that $\ve$ represents the time scale for convergence of $A_\ve$ to equilibrium. Recalling that $A_\ve$ is the normal component (to leading order) of the actin filament orientation and that $\ve$ represents the width of the diffuse interface where actin polymerization occurs, we interpret $\ve$ as the scale of inertial forces of actin filament protrusion.

In order to numerically simulate the system \eqref{numeric_veqn}-\eqref{numeric_aeqn} we first truncate the domain of $A_\ve(z,s,t)$ to $z\in [-L,L]$ and assume Dirichlet boundary conditions $A_\ve(\pm L, s,t)=0$. For subsequent numerical simulations we take $L=20$. Due to exponential decay of $A_\ve(z,\cdot,\cdot)$ this approximation introduces negligible error. We use centered finite differences to approximate spatial derivatives and an (explicit) forward in time discretization for \eqref{numeric_veqn}. However, due to the singular perturbation in the time derivative of $A_\ve$, it is beneficial to use backward in time discretizations in \eqref{numeric_aeqn} as it yields an unconditionally stable algorithm. We use the Thomas algorithm to invert the resultant tridiagonal system. 

We recall the following standard notations. Let $p_i = (x_i,y_i)$, $i=1,\dots, N$ be a discretization of a curve. Then $h:=1/N$ is the grid spacing and the first and second derivatives are defined
\begin{equation}
Dp_i := \frac{-p_{i+2}+8p_{i+1}-8p_{i-1}+p_{i-2}}{12h}\text{ and  } D^2p_i := \frac{-p_{i+2}+16p_{i+1}-30p_i+16p_{i-1}-p_{i-2}}{12h^2} \;\;\;
\end{equation}
with $p_{-j} := p_{N-j}$ due to periodicity.

Our algorithm is modified from the algorithm introduced in \cite{MizBerRybZha15}: \\
\begin{algorithm}
\begin{enumerate}
\item[(I)] (Pre-computation) Given a double-well potential, $W=W(z)$, solve for $\theta_0$ by solving
\begin{equation}\label{new_theta}
\theta_0'(z) = \sqrt{2W(\theta_0(z))},\;\; \theta_0(0)=\frac{1}{2}.
\end{equation}
This problem is equivalent to \eqref{standingwave} but avoids solving a boundary value problem at $\pm \infty$.

\item[(II)] (Initialization) Input a closed curve $\Gamma$ discretized by $N$ points $s^0_i=(x^0_i,y^0_i)$. For each point $s^0_i$ there is an associated function $A(s^0_i,z)$ on the interval $[-L,L]$ discretized by $a^0_{i,j}$, $j\in \{1,\dots,M\}$. Then $\Delta z_M = 2L/M$ is the space step on the interval $[-L,L]$. For all time, fix $\theta_0'$ which is discretized by $(\theta_0')_j$.  


Use the shoelace formula to calculate the area of $\Gamma(0)$:
\begin{equation}\label{eqn:shoelace}
A^o = \frac{1}{2} \left| \sum_{i=1}^{n-1} x^0_i y^0_{i+1} + x^0_n y^0_1 - \sum_{i=1}^{n-1} x^0_{i+1}y^0_i - x^0_1 y^0_n \right|.
\end{equation}

\item[(III)] (time evolution) Calculate the curvature at each point, $\kappa_i$ using the formula
\begin{equation}\label{curvature_eqn}
\kappa_i = \frac{\operatorname{det}(Dp^t_i,D^2 p^t_i)}{\|Dp^t_i\|^3},
\end{equation}
where $\|\cdot\|$ is the standard Euclidean norm.
Compute
\begin{equation}\label{phi_def}
\Phi_\beta^i = \sum_{j} a_{i,j}\cdot (\theta_0')_j^2 \Delta x_M
\end{equation}
and compute $\lambda = \frac{1}{|\Gamma|}\int_{\Gamma} \kappa + \Phi_\beta ds$ using a trapezoidal rule and the discretizations \eqref{curvature_eqn}-\eqref{phi_def}. Define
\begin{equation}\label{discreteV}
V^{temp}_i = \kappa_i+\Phi_\beta^i -\lambda.
\end{equation}
Define the temporary curve
\begin{displaymath}
p^{temp}_i := p^t_i+V^{temp}_i \nu_i\Delta t,
\end{displaymath}
where $\nu_i=(D s_i^0)^\perp/\|Ds_i^0\|$ is the inward pointing normal vector. \\

Use Thomas algorithm to update $a^0_{i,j}\mapsto a^{\Delta t}_{i,j}$ for each $i$.

\item[(IV)] (area adjustment) Calculate the area of the temporary curve $A^{temp}$ using the shoelace formula \eqref{eqn:shoelace} and compute the discrepancy
\begin{displaymath}
\Delta A:=(A^{temp}-A^{o})\cdot (A^o)^{-1},
\end{displaymath}
If $|\Delta A|$ is larger than a fixed tolerance $\delta>0$, adjust $\lambda \mapsto \lambda +\Delta A$ and solve \eqref{discreteV} with updated $\lambda$. Otherwise define $p_i^{\Delta t} := p_i^{temp}$ and
\begin{displaymath}
\Gamma(\Delta t):=\{p_i^{\Delta t}\}.
\end{displaymath}
\end{enumerate}
\end{algorithm}

We highlight that the area enclosed by the curve may have large deviation in long time simulation due to accumulation of errors. Part $(IV)$ is introduced precisely to punish such changes in the area. Simulations by this algorithm qualitatively agree with those conducted in \cite{MizBerRybZha15} in the case that $\beta$ is subcritical. We implement the above algorithm in C++ and visualize the data using Scilab. We choose the time step $\Delta t$ and spatial discretization step $\displaystyle h = \frac{1}{N}$, so that
\begin{equation}
\frac{\Delta t}{h^2} = O(10^{-3}),
\end{equation}
in order to ensure stability.
Further, we take an error tolerance $err = 10^{-8}$ in Part $(IV)$.


In the sequel, we consider the asymmetric potential \eqref{asym_pot_well}.  In particular, using $\beta=100$ we use the traveling wave profile depicted in Figure \ref{figs:real_profile} as initial condition in \eqref{numeric_veqn}-\eqref{numeric_aeqn}. Taking the traveling wave velocity $V\approx 1.7$, we initialize $A_\ve(s,z,0)$ to be the solution to 
\begin{equation}
0 = \partial_{zz} A_\ve + V_\ve \partial_z A - A -\beta \theta_0',
\end{equation}
where we emphasize that $V_\ve=V_\ve(s)$ is the normal velocity of the traveling wave solution at each point of the curve and not the total velocity.

\subsection{Simulation results}
Simulations of \eqref{numeric_veqn}-\eqref{numeric_aeqn} show that all traveling wave solutions (including steady circles) are unstable: a small perturbation (e.g., due to numerical error) of any traveling wave solution results in large deviations in the curve profile. In general we observe: \\

\begin{observation}
Traveling wave solutions of \eqref{numeric_veqn}-\eqref{numeric_aeqn} are unstable.
\end{observation}

\begin{observation}
The long time behavior of curves evolving by \eqref{numeric_veqn}-\eqref{numeric_aeqn} is (i) rotating solutions if $\ve < .005$ and (ii) wandering cells if $\ve>.005$.\\
\end{observation}

{\em Behavior (i): Rotations}. If $\ve < .005$, then traveling wave solutions immediately change shape and exhibit a periodic wave of protrusion which laterally traverses the curve. These curves appear to be near circular with a rotating protrusion; we call such solutions rotating solutions (see Figure \ref{figs:ratio} for a sketch).  We conjecture that this behavior is governed by stable/unstable velocities $V_\ve$ (as defined in Section \ref{pf_intro}): on small intervals on either side of the protrusion wave there are unstable velocities. These are sketched in Figure \ref{figs:ratio} as red intervals. The fact that simulations suggest that all traveling wave solutions are unstable is expected since linear stability analysis of the 1D intermediate problem \eqref{stability_eqn1}-\eqref{stability_eqn2} established that a velocity $V$ is stable if and only if $\Phi_\beta'(V)<c_0$ \cite{BerPotRyb16}. Using this criterion and our choice of $\Phi_\beta$, it is the case that $V=0$ is unstable; since there necessarily must be at least two points on a traveling wave curve with normal velocity $V=0$, intervals of unstable velocities are expected in the 2D traveling waves.  Indeed fixing a point on the curve and tracking its velocity over time we observe a hysteresis phenomenon similar to the 1D case \cite{BerPotRyb16}. We plot the results in Figure \ref{figs:hysteresis_zigzag}.

We observe that as $\ve\to 0$ then the resultant rotation solutions have smaller protrusions and smaller period.  To quantify this we plot the isoperimetric inequality
\begin{equation}
Q = \frac{4\pi Area(\Gamma)}{|\Gamma|^2}.
\end{equation}
of several rotating solutions over time in Figure \ref{figs:ratio}.



\begin{figure}[h]
\includegraphics[width = \textwidth]{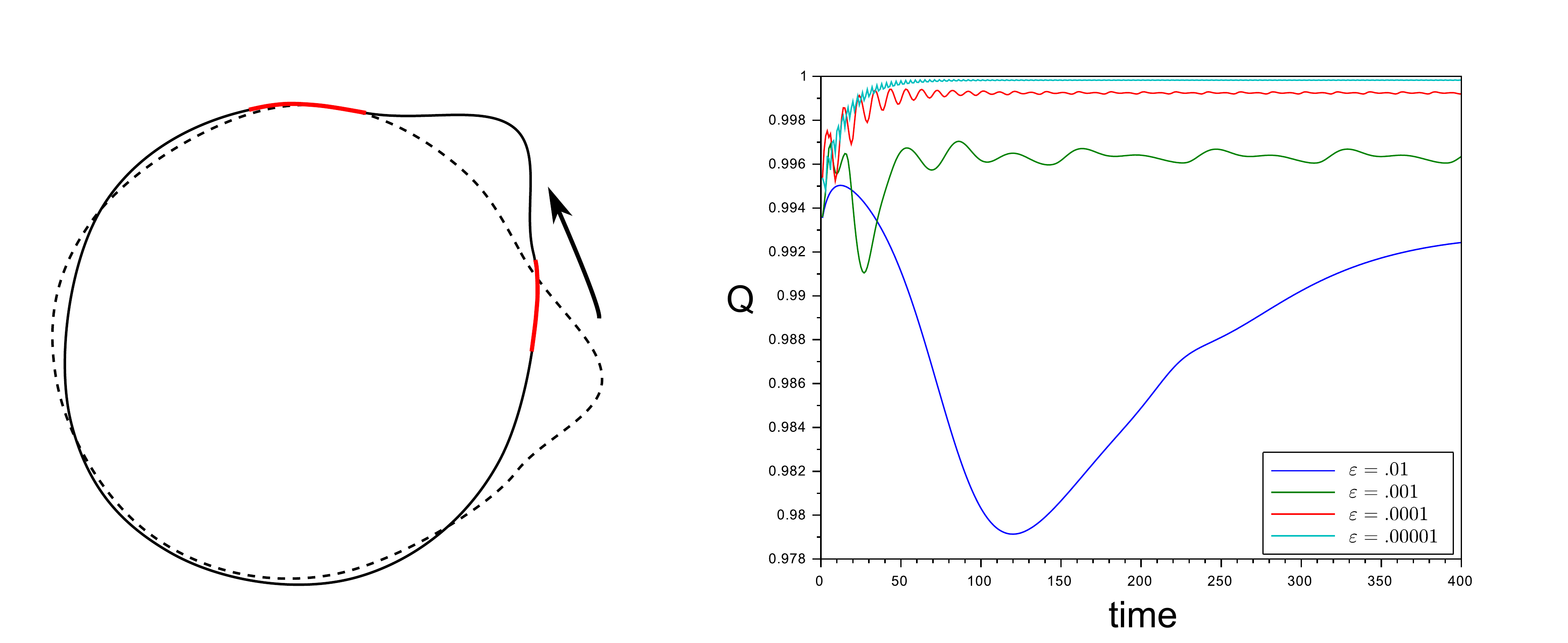}
\caption{(Left) Sketch of rotating cell; intervals in red represent unstable velocities (Right) Graph of the isoperimetric inequality $Q$ over time for various $\ve$}
\label{figs:ratio}
\end{figure}

The isoperimetric quotient is a measure of how far a curve is from a perfect circle: in general $Q\leq 1$ and $Q= 1$ if and only if $\Gamma$ is a circle. We indeed observe that as $\ve\to 0$ that the resulting periodic solutions converge to circles and that the period and amplitude both decrease as $\ve \to 0$. 
These rotating curves are in qualitative agreement with recent experimental data which observed rotating cells as a result of myosin activation \cite{Lou15}. In particular our simulations capture the experimental finding that protrusion size is correlated to protrusion lifetime, i.e., duration of time that a region remains protruded. \\ 

{\em Behavior (ii): Wandering cells}. If $\ve > .005$ then traveling wave solutions evolving by \eqref{numeric_veqn}-\eqref{numeric_aeqn} travel several times their own lengths before perturbations in the cell shape result in a turning or wandering cell. For example, in the case that $\ve = .01$ the cell undergoes a period of transitional turning and then walks in ``zig-zags'' indefinitely, see Figure \ref{figs:hysteresis_zigzag}. These simulations are similar to the bipedal motion described in \cite{LobZieAra14, BarAllJulThe10}, wherein cells undergo shape oscillations resulting in non-straight line trajectories.

As suggested by Figure \ref{figs:ratio} the period of the intervals of unstable velocities is much longer than in the case that $\ve < .005$ resulting in the speed of curve deformation to be of similar order to the speed of the intervals of unstable velocities. We conjecture that the balance of these two effects are responsible for the turning and wandering behavior of cells. Indeed in \cite{BarAllJulThe10} the authors suggest a hysteresis loop as a mechanism for bipedal motion very similar to the results of Figure \ref{figs:hysteresis_zigzag}.


\begin{figure}
\includegraphics[width = \textwidth]{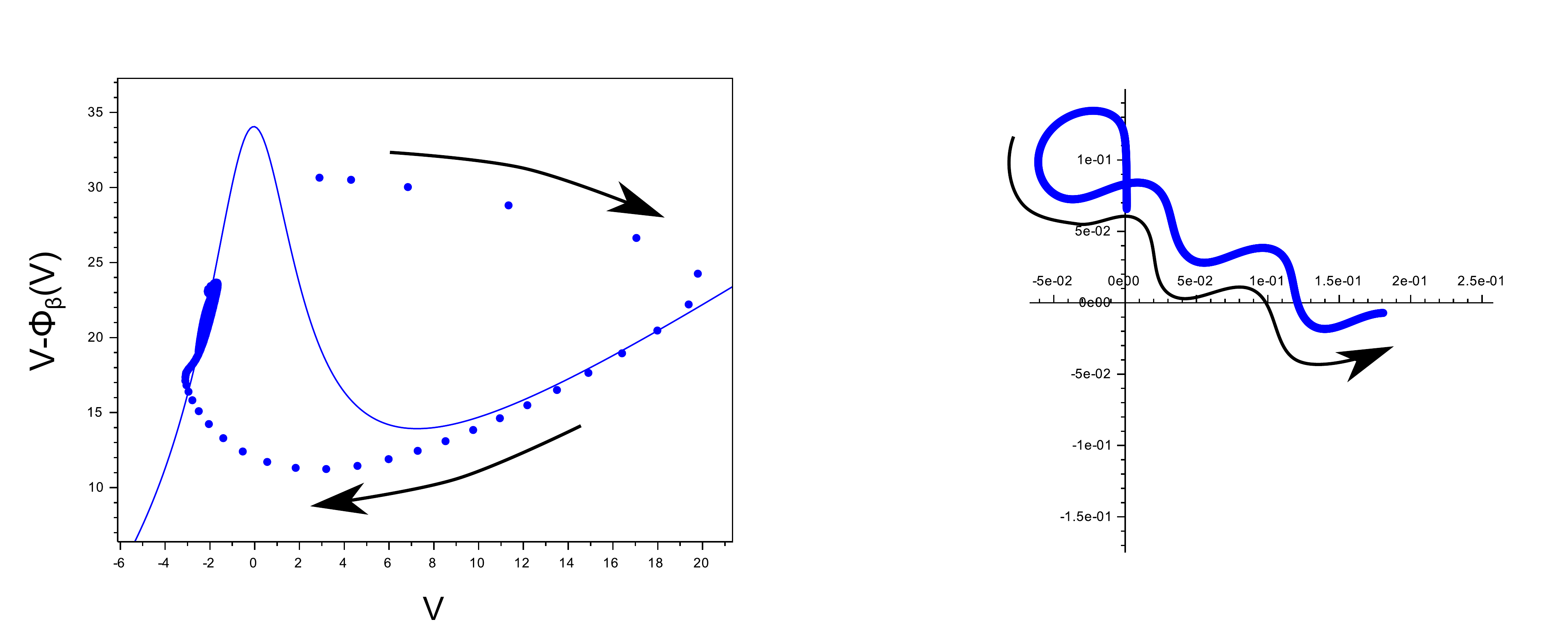}
\caption{(Left) Plot of $V-\Phi_\beta(V)$ tracked for a point on a curve evolving by system \eqref{numeric_veqn}-\eqref{numeric_aeqn} shows approximate hysteresis jumps. (Right) Trajectory of center of curve in \eqref{numeric_veqn}-\eqref{numeric_aeqn} when $\ve = .01$; after short transience period, convergence to ``zig-zag'' motion}
\label{figs:hysteresis_zigzag}
\end{figure}

{\em Biological interpretation:} Recall that $\ve$ represents the inertial scale of actin filament protrusion. Our results predict that $\ve$ must be sufficiently large to induce persistent bipedal motion. Otherwise the membrane protrusion lifetime generated by actin filament polymerization is too small to overcome visco-elastic membrane effects. In this latter case, the short lifetime results in lateral waves of protrusion (as observed in \cite{Lou15}).

We predict that the lack of (numerically) stable traveling wave solutions is due to the lack of spatially varying myosin contractility in the SIL.  The effects of myosin are incorporated in \eqref{eq1}-\eqref{eq2} via an asymmetric double well potential $W$ leading to an asymmetric function $\Phi_\beta$. By varying $\Phi_\beta$ along the curve we can incorporate spatially varying myosin effects, e.g., in experiments myosin contractility is more prevalent in the rear part of the cell than in the front (lamellipod).

\section{Conclusion}

We have presented both analytical and numerical results of the geometric evolution equation \eqref{interface} arising as the sharp interface limit of a phase-field model for crawling cell motility \eqref{eq1}-\eqref{eq2}. The key biophysical parameter $\beta$ captures the effect of actin polymerization strength and adhesion strength. In the case of subcritical $\beta$ the equation \eqref{interface} is uniquely solvable for normal velocity and thus is amenable to analytical study.  We have proved uniqueness of solutions of \eqref{interface} relying on a Gr\"onwall estimate of specially weighted $L^2$ norms. Moreover, we have proved that for both subcritical and supercritical $\beta$ regimes, if the nonlinearity $\Phi_\beta(V)$ is an even function then no traveling wave solutions to \eqref{interface} exist. However if $\Phi_\beta'(0)>0$ then we have proved that traveling waves exist for sufficiently large supercritical $\beta$. These results are crucial to understand the persistence of motile cells in experiments. Numerical simulation of \eqref{interface} in the case of supercritical $\beta$ requires a selection criterion from multiple solutions due to non-unique solvability of normal velocity. As such, we utilize an intermediate equation representative of both the phase-field model and the sharp interface limit equation. Our simulations revealed two phenomena which are both experimentally relevant: lateral protrusion waves (rotating cells) and bipedal cell motion (wandering cells) depending on the value of $\ve$, representing the inertial scale of actin protrusion forces.


An open question is to analytically study the stability of traveling waves, e.g., via linear stability analysis. Linear stability analysis in 2D is much more difficult due to coupling of geometry (i.e. curvature) with the dynamic PDE \eqref{numeric_aeqn}. We conducted preliminary numerical investigation of the linearized operator of \eqref{numeric_veqn}-\eqref{numeric_aeqn} around circular solutions. Using finite difference methods we found that some eigenvalues of the finite-dimensional approximation of the linearized operator have positive real part.  This is evidence of instability of circular steady states yet it is well known that convergence of the spectrum of finite-dimensional approximations to the full linearized operator is not well-behaved (see e.g., \cite{TreEmb05}). Rigorous spectral analysis will be a future work of the authors.

{\bf Acknowledgments.} The authors would like to thank Leonid Berlyand and Lei Zhang for their hospitalities, guidance, and fruitful discussions.

%
%
%
%

\bibliographystyle{siam}
\bibliography{cellref}

\end{document}